\documentclass[11pt,reqno]{amsart}
\usepackage{fixltx2e}                    
\usepackage{mathpazo}  
\usepackage[utf8]{inputenc}            
\usepackage{amsmath}                     
\usepackage{amssymb, latexsym, stmaryrd, amsthm, dsfont, amsfonts, amsbsy,amsthm, amsmath, mathrsfs}            
\usepackage{mathtools}                   
\usepackage{bm}                          
\usepackage{enumerate}                   
\usepackage{verbatim}                    
\usepackage{url}   
\usepackage{lscape}                      
\usepackage{soul}

\usepackage[margin=1.2in]{geometry}

\usepackage{microtype}  
\usepackage[all, knot]{xy}
        \xyoption{arc} 
        \xyoption{web}                 
\makeatletter                            
\def\MT@register@subst@font{\MT@exp@one@n\MT@in@clist\font@name\MT@font@list
 \ifMT@inlist@\else\xdef\MT@font@list{\MT@font@list\font@name,}\fi}
\makeatother

\usepackage[pdftex,bookmarks,bookmarksnumbered,linktocpage,   
         colorlinks,linkcolor=blue,citecolor=blue]{hyperref}

\newcommand{\bit}{\begin{itemize}}    
\newcommand{\eit}{\end{itemize}}
\newcommand{\ben}{\begin{enumerate}}
\newcommand{\een}{\end{enumerate}}

\newcommand{\benroman}{\ben[\normalfont (i)]}  
\let\eroman\een

\newcommand{\bde}{\begin{description}}
\newcommand{\ede}{\end{description}}

\newcommand{\?}{\ensuremath{\mkern0.4\thinmuskip}}   

\let\leq=\leqslant
\let\geq=\geqslant

\let\epsilon=\varepsilon
\let\Lambda\varLambda
\let\Gamma\varGamma
\let\Delta\varDelta
\let\Lambda\varLambda
\let\Omega\varOmega
\let\Theta\varTheta
\let\Xi\varXi
\let\Pi\varPi
\let\Sigma\varSigma

\let\oper=\mathbb                               

\bmdefine{\A}{A}                                
\bmdefine{\2}{2}
\bmdefine{\B}{B}
\bmdefine{\D}{D}
\bmdefine{\M}{M}                                
\bmdefine{\LLL}{L}                              
\bmdefine{\Fm}{Fm}                              
\bmdefine{\zerou}{[0{,}1]}  
\bmdefine{\T}{T}                                

\newcommand{\Lim}{\oper{L}_{\textsc{c}}^{}}

\newcommand{\PPU}{\oper{P}_{\!\textsc{u}}^{}}
\newcommand{\PPR}{\oper{P}_{\!\textsc{r}}^{}}

\newcommand{\PPp}{\oper{P}_{\!\textsc{p}}^{}}
\newcommand{\PPF}{\oper{P}_{\!\textsc{f}}^{}}

\bmdefine{\boldstar}{\mathchoice{\textstyle*}{\textstyle*}{\textstyle*}{\scriptstyle*}}

\bmdefine{\btau}{\tau}                                  
\bmdefine{\brho}{\rho}                                  

\usepackage{todonotes}

\bmdefine{\leibniz}{\Omega}        

\bmdefine{\frege}{\Lambda}         

\makeatletter
\newcommand{\tarskidsp}{\mathord%
   {\m@th\raisebox{0pt}[0pt][0pt]{$\stackrel%
   {\raisebox{-2.7pt}[0ex][0pt]{$\displaystyle \,\?\thicksim$}}%
   {\displaystyle\leibniz}$}}}
\newcommand{\tarskitxt}{\mathord%
   {\m@th\raisebox{0pt}[0pt][0pt]{$\stackrel%
   {\raisebox{-2.7pt}[0ex][0pt]{$\,\?\thicksim$}}{\displaystyle\leibniz}$}}}
\newcommand{\tarskiscr}{\mathord%
   {{\m@th\raisebox{0pt}[0pt][0pt]{$\stackrel%
   {\raisebox{-2.4pt}[0ex][0pt]{$\scriptstyle \,\?\thicksim$}}%
   {\scriptstyle\leibniz}$}}}}
\newcommand{\tarskiscrscr}{\mathord%
   {{\m@th\raisebox{0pt}[0pt][0pt]{$\stackrel%
   {\raisebox{-2pt}[0ex][0pt]{$\scriptscriptstyle \,\?\thicksim$}}%
   {\scriptscriptstyle\leibniz}$}}}}
\newcommand{\tarski}{\@ifnextchar ^ %
   {\mathchoice{\tarskidsp\kern-.07em}{\tarskitxt\kern-.07em}%
   {\tarskiscr\kern-.07em}{\tarskiscrscr\kern-.07em}}%
   {\mathchoice{\tarskidsp}{\tarskitxt}{\tarskiscr}{\tarskiscrscr}}}
\makeatother

\theoremstyle{theorem}
\newtheorem{Theorem}{Theorem}[section]
\newtheorem{Proposition}[Theorem]{Proposition}

\newtheorem{Corollary}[Theorem]{Corollary}
\newtheorem{Claim}[Theorem]{Claim}

\newtheorem{Positive Los Theorem}[Theorem]{Positive \L os Theorem}
\newtheorem{Positive Keisler Isomorphism Theorem}[Theorem]{Positive Keisler Isomorphism Theorem}

\theoremstyle{definition}
\newtheorem{law}[Theorem]{Definition}
\newtheorem{exa}[Theorem]{Example}

\theoremstyle{remark}

\newtheorem{Remark}[Theorem]{Remark}

\DeclareMathOperator{\Res}{Res}
\DeclareMathOperator{\Th}{Th}

\newcommand\er[1]{\mathord{\equiv_{#1}}}


\begin{document}
\title[Elementary equivalence in positive logic via prime products]{Elementary equivalence in positive logic via prime products}

\author{T. Moraschini, J.J. Wannenburg, and K. Yamamoto}

\address{Tommaso Moraschini: Departament de Filosofia, Facultat de Filosofia, Universitat de Barcelona (UB), Carrer Montalegre, $6$, $08001$ Barcelona, Spain}
\email{tommaso.moraschini@ub.edu}

\address{Johann J.\ Wannenburg: \'{U}stav informatiky Akademie v\v{e}d \v{C}esk\'{e} republiky, Pod Vod\'{a}renskou v\v{e}\v{z}\'{i} 2, 182 07 Praha 8, the Czech Republic
\and
School of Mathematics, University of the Witwatersrand, South Africa
}\email{jamie.wannenburg@up.ac.za}

\address{Kentaro Yamamoto: \'{U}stav informatiky Akademie v\v{e}d \v{C}esk\'{e} republiky, Pod Vod\'{a}renskou v\v{e}\v{z}\'{i} 2, 182 07 Praha 8, the Czech Republic}\email{yamamoto@cs.cas.cz}

\date{\today}

\begin{abstract}
We introduce \emph{prime products} as a generalization of ultraproducts for positive logic.\ Prime products are shown to satisfy a version of \L o\'s's Theorem restricted to  positive formulas, as well as the following variant of the Keisler Isomorphism Theorem: under the generalized continuum hypothesis, two models have the same positive theory  if and only if  they have isomorphic prime powers of ultrapowers.
\end{abstract}

\maketitle

\section{Introduction}

A map $f \colon M \to N$ between two structures $M$ and $N$ is said to be a \emph{homomorphism} when for every atomic formula $\varphi(x_1, \dots, x_n)$ and $a_1, \dots, a_n \in M$,
\[
M \vDash \varphi(a_1, \dots, a_n) \text{ implies }N \vDash \varphi(f(a_1), \dots, f(a_n)).
\]
Positive model theory is the branch of model theory that deals with the formulas that are preserved by homomorphisms (see, e.g., \cite{Poizat06a,Poizat10b,PoizatYeshkeyev,BenYac03,BenYac03b,BenYacPoiz07}). It is well known that these are precisely the \emph{positive formulas}, that is, the formulas built from atomic formulas and $\bot$ using only $\exists, \land$, and $\lor$. 

The Keisler-Shelah Isomorphism Theorem states that two structures are elementarily equivalent  if and only if  they have isomorphic ultrapowers. This celebrated result was first proved  by Keisler under the generalized continuum hypothesis (GCH) \cite[Thm.\ 2.4]{Ke61}. This assumption was later shown to be redundant by Shelah \cite[p.\ 244]{She71}. The aim of this paper is to prove a version of Keisler's original theorem in the context of positive model theory. 

To this end, we say that two structures are \emph{positively equivalent} when they have the same positive theory. In order to obtain a positive version of Keisler Isomorphism Theorem, we will introduce a generalization of the ultraproduct construction that captures positive equivalence. We term this construction a \emph{prime product} because it is obtained by replacing the index set $I$ typical of an ultraproduct by a poset $\mathbb{X}$ and the ultrafilter over $I$ by a \emph{prime} filter of the bounded distributive lattice of upsets of the poset $\mathbb{X}$. The case of traditional ultraproducts is then recovered by requiring the order of $\mathbb{X}$ to be the identity relation.

Prime products and positive formulas are connected by the natural incarnation of \L os Theorem in this context (Theorem \ref{Positive Los Theorem}). As a consequence, prime products preserve not only positive formulas, but also the universal closure of the implications between them, known as \emph{basic h-inductive sentences} \cite{PoizatYeshkeyev} (Proposition \ref{Prop : prime products preserve h-inductive}). This allows us to describe the classes of models of h-inductive theories as those closed under isomorphisms, prime products, and ultraroots (Corollary \ref{Cor : h-inductive classes via prime products}).

Our main result states that under GCH two structures have the same \emph{positive theory}  if and only if  they have isomorphic \emph{prime powers} of ultrapowers (Theorem \ref{Positive Keisler Isomorphism Theorem}). The same result holds without GCH, provided that prime powers are replaced by \emph{prime products} in the statement (Theorem \ref{thm:old}). Notably, the presence of ultrapowers cannot be removed from this theorems, as there exists positively equivalent structures without isomorphic prime powers (Example \ref{Exa : positive equivalent without prime powers}).

\section{Prime products}

A subset $V$ of a poset $\mathbb{X} = \langle X; \leq \rangle$ is said to be an \emph{upset} when for every $x, y \in X$,
\[
\text{if }x \in V\text{ and }x \leq y\text{, then }y \in V.
\]
The \emph{downsets} of $\mathbb{X}$ are defined dually. An upset $V$ of $\mathbb{X}$ is \emph{proper} when it differs from $X$, and it is \emph{principal} when it coincides with
\[
{\uparrow}x \coloneqq \{ y \in X \mid x \leq y \},
\]
for some $x \in X$. When ordered under the inclusion relation, the family $\mathsf{Up}(\mathbb{X})$ of upsets of $\mathbb{X}$ forms a bounded distributive lattice
\[
\mathsf{Up}(\mathbb{X}) \coloneqq \langle \mathsf{Up}(\mathbb{X}); \cap, \cup, \emptyset, X \rangle.
\]

\begin{law}
A \emph{filter} over a poset $\mathbb{X}$ is a nonempty upset of the lattice $\mathsf{Up}(\mathbb{X})$ which, moreover, is closed under binary intersections. In this case, $F$ is said to be \emph{prime} when it is proper and for every $V, W \in \mathsf{Up}(\mathbb{X})$,
\[
V \cup W \in F \text{ implies that either }V \in F \text{ or }W \in F.
\]
\end{law}

\begin{Remark}\label{Rem: filters on X}
Given a set $X$, we denote the poset whose universe is $X$ and whose order is the identity relation by $\mathsf{id}(X)$. In this case, $\mathsf{Up}(\mathsf{id}(X)) = \mathcal{P}(X)$. Furthermore, the filters (resp.\ prime filters) over $\mathsf{id}(X)$ coincide with the filters (resp.\ ultrafilters) over the set $X$.
\qed
\end{Remark}

An \emph{ordered system} (of structures) comprises a nonempty family $\{ M_x \mid x \in X \}$ of similar structures indexed by a poset $\mathbb{X}$ and a family of homomorphisms $\{ f_{xy} \colon M_x \to M_y \mid x, y \in X \text{ and }x \leq y \}$ such that $f_{xx}$ is the identity map on $M_x$ and for every $x, y, z \in X$,
\[
x \leq y \leq z \text{ implies }f_{xz} = f_{yz}\circ f_{xy}.
\]
A poset is said to be a \emph{wellfounded forest} when its principal downsets are well ordered.

We will associate a new structure with every ordered system $\{ M_x \mid x \in X \}$ indexed by a wellfounded forest $\mathbb{X}$ and every filter $F$ over $\mathbb{X}$ as follows.  First, for every $V \in F$ let
    \begin{align*}
S_V \coloneqq \{ a \in \prod_{x \in V} M_x \mid  \text{for every }y, z \in V,  \text{ }y \leq z \text{ implies } f_{yz}(a(y)) = a(z)\}
    \end{align*}
   and consider the union
    \[
    S_F \coloneqq \bigcup_{V \in F}S_V.
    \]
 Then for every $a \in S_F$ let $V_a$ be the domain of $a$, that is,
\[
V_a \coloneqq \text{the unique }V \in F\text{ such that }a \in S_V.
\]
It will often be convenient to restrict the sequence $a$ to some $V \in F$ such that $V\subseteq V_a$ as follows:
\[
a{\upharpoonright_{V}} \coloneqq \langle a(x) \mid x \in V \rangle. 
\]
Notice that from $a \in S_F$ it follows that $a_{\upharpoonright_{V}} \in S_V$. Lastly, for every formula $\varphi(v_1, \dots, v_n)$ and $a_1, \dots, a_n \in S_F$ let
\[
\llbracket \varphi(a_1, \dots, a_n) \rrbracket \coloneqq \{ x \in X \mid x \in V_{a_1} \cap \dots \cap V_{a_n} \text{ and }M_x \vDash \varphi(a_1(x), \dots, a_n(x)) \}.
\]
We define an equivalence relation on $S_F$ as follows: for every $a, b \in S_F$,
\[
{a} \equiv_{F} {b} \, \, \text{  if and only if  } \, \, \llbracket a = b \rrbracket \in F.
\]
The proof of the following observation is a routine exercise.

\begin{Proposition}\label{Prop: Prime products are well defined}
Let $g$ be a basic $n$-ary operation, $R$ a basic $n$-ary relation, $a_1, \dots, a_n, b_1, \dots, b_n \in S_F$, and
\[
V_a \coloneqq V_{a_1} \cap \dots \cap V_{a_n} \, \, \text{ and } \, \, V_b \coloneqq V_{b_1} \cap \dots \cap V_{b_n}.
\]
Then $g^{\prod_{x \in V_a}M_x}(a_1{\upharpoonright_{V_a}}, \dots, a_n{\upharpoonright_{V_a}}), g^{\prod_{x \in V_b}M_x}(b_1{\upharpoonright_{V_b}}, \dots, b_n{\upharpoonright_{V_b}}) \in S_F$. Moreover, if ${a_m} \equiv_{F} {b_m}$ for every $m \leq n$, then
\[
\llbracket \? g^{\prod_{x \in V_a}M_x}(a_1{\upharpoonright_{V_a}}, \dots, a_n{\upharpoonright_{V_a}}) = g^{\prod_{x \in V_b}M_x}(b_1{\upharpoonright_{V_b}}, \dots, b_n{\upharpoonright_{V_b}}) \rrbracket \in F
\]
and
\[
\llbracket R(a_1, \dots, a_n) \rrbracket  \in F \, \, \text{  if and only if  }\, \, \llbracket R(b_1, \dots, b_n) \rrbracket  \in F.
\]
\end{Proposition}

In view of Proposition \ref{Prop: Prime products are well defined}, the following structure is well defined:

\begin{law}
The \emph{filter product} $\prod_{x \in X}M_x / F$ is the structure with universe $S_F / {\equiv_F}$ where
\benroman
\item  the basic $n$-ary operations $g$ are defined as
\[
g(a_1/ {\equiv_{F}}, \dots, a_n/ {\equiv_{F}}) \coloneqq g^{\prod_{x \in V_a}M_x}(a_1{\upharpoonright_{V_a}}, \dots, a_n{\upharpoonright_{V_a}}) / {\equiv_{F}}
\]
for $V_a \coloneqq V_{a_1} \cap \dots \cap V_{a_n}$; 
\item the basic $n$-ary relations $R$ are defined as
\begin{align*}
\langle a_1 / {\equiv_{F}}, \dots, a_n / {\equiv_{F}} \rangle \in R \Longleftrightarrow  \llbracket R(a_1, \dots, a_n) \rrbracket  \in F.
\end{align*}
\eroman
When each $M_x$ is the same structure $M$, we say that $M^X / F$ is a \emph{filter power} of $M$.
\end{law}

Typical examples of filter products include reduced products.

\begin{exa}[\textsf{Reduced products}]\label{Exa : reduced products}
Recall from Remark \ref{Rem: filters on X} that the filters over a set $X$ coincide with the  filters over the poset $\mathsf{id}(X)$. We will show that the reduced product $M$ of a family $Y = \{ M_x \mid x \in X \}$ of similar structures induced by a filter $F$ over $X$ is isomorphic to the filter product of $Y$, viewed as an ordered system indexed by the wellfounded forest $\mathsf{id}(X)$, induced by the same filter $F$ over $\mathsf{id}(X)$.

To this end, let $\mathsf{D}(Y, F)$ be the ordered system comprising the family of structures $\{ \prod_{x \in V}M_x \mid V \in F \}$ indexed by the poset $\langle F; \supseteq \rangle$ and the canonical projections $f_{V,W} \colon \prod_{x \in V} M_x \to \prod_{x \in W}M_x$ for each $V, W \in F$ with $W \subseteq V$. As $F$ is closed under binary intersections, $\mathsf{D}(Y, F)$ is a direct system. Furthermore, its direct limit coincides with the filter product of $Y$ induced by $F$. This is because $Y$ is indexed by $\mathsf{id}(X)$ and, therefore,
\[
S_F = \bigcup \{ \prod_{x \in V} M_x \mid V \in F \}
\]
and for every $a, b \in S_F$,
\[
{a} \equiv_F {b} \, \, \text{  if and only if  }\, \, f_{V_a, V_a\cap V_b}(a) = f_{V_b, V_a\cap V_b}(b).
\]

As the direct limit of $\mathsf{D}(Y, F)$ is isomorphic to the reduced product of $Y$ induced by $F$ (see, e.g., \cite[p.\ 109]{Eklof77}), we conclude that so is the filter product of $Y$ induced by $F$.\qed
\end{exa}

The following construction plays the role of an ultraproduct in positive model theory.

\begin{law}
A filter product $\prod_{x \in X} M_x / F$ is said to be a \emph{prime product} when $F$ is prime. If, in addition, each $M_x$ is the same structure $M$, we say that $M^X / F$ is a \emph{prime power} of $M$.
\end{law}

\begin{exa}[\textsf{Ultraproduct}]
Recall from Remark \ref{Rem: filters on X} that the ultrafilters over a set $X$ coincide with the prime filters over the poset $\mathsf{id}(X)$. Therefore, Example \ref{Exa : reduced products} shows that the ultraproduct of a family $Y = \{ M_x \mid x \in X \}$ of similar structures induced by an ultrafilter $U$ over $X$ is isomorphic to the prime product of $Y$, viewed as an ordered system indexed by $\mathsf{id}(X)$, induced by the prime filter $U$ over $\mathsf{id}(X)$.
\qed
\end{exa}

Let $\{ M_x \mid x \in X \}$ be an ordered system indexed by a linearly ordered poset $\mathbb{X}$. It is easy to construct a nonempty well ordered subposet $\mathbb{Y}$ of $\mathbb{X}$ that is cofinal in $\mathbb{X}$, i.e., such that
\[
\text{for every }x \in X \text{ there exists }y \in Y \text{ such that }x \leq y.
\]
\noindent Furthermore, when viewed as an ordered system, $\{ M_y \mid y \in Y \}$ has the same direct limit as $\{ M_x \mid x \in X \}$. Because of this, when discussing ordered systems indexed by a linearly ordered poset $\mathbb{X}$, we will restrict our attention to the case where $\mathbb{X}$ is well ordered. Accordingly, by a \emph{chain of structures} we understand an ordered system $\{ M_x \mid x \in X \}$ indexed by a well ordered poset $\mathbb{X}$. A simple example of a prime product is the direct limit of a chain of structures, as we proceed to illustrate.

\begin{exa}[\textsf{Limits of chains}]\label{Exa : limits of chains}
Let $\{ M_x \mid x \in X \}$ be a chain of structures indexed by $\mathbb{X}$. Since $\mathbb{X}$ is nonempty and linearly ordered, $F \coloneqq \{ {\uparrow}x \mid x \in X \}$ is a filter over $\mathbb{X}$. Moreover, the direct limit of $\{ M_x \mid x \in X \}$ is isomorphic to the prime product $\prod_{x \in X}M_x / F$.
\qed
\end{exa}

For every class $\mathsf{K}$ of similar structures let
\begin{align*}
\PPR(\mathsf{K}) &\coloneqq \text{the class of structures isomorphic to a reduced product of members of }\mathsf{K};\\
\PPU(\mathsf{K}) &\coloneqq \text{the class of structures isomorphic to a ultraproduct of members of }\mathsf{K};\\
\PPF(\mathsf{K}) &\coloneqq \text{the class of structures isomorphic to a filter product of members of }\mathsf{K};\\
\PPp(\mathsf{K}) &\coloneqq \text{the class of structures isomorphic to a prime product of members of }\mathsf{K};\\
\Lim(\mathsf{K}) &\coloneqq \text{the class of structures isomorphic to the direct limit of a chain of members of  }\mathsf{K}.
\end{align*}
The proof of the following technical observation is contained in the Appendix.

\begin{Proposition}\label{Prop : ultraproducts of chains}
For every class $\mathsf{K}$ of similar structures,
\[
\PPR\Lim(\mathsf{K}) \subseteq \PPF(\mathsf{K}) \, \, \text{ and } \, \, \PPU\Lim(\mathsf{K}) \subseteq \PPp(\mathsf{K}).
\] 
\end{Proposition}

The importance of prime products derives from the following observation:

\begin{Positive Los Theorem}\label{Positive Los Theorem}
Let $\prod_{x \in X}M_x / F$ be a prime product. For every positive formula $\varphi(v_1, \dots, v_n)$ and $a_1, \dots, a_n \in S_F$,
\[
\prod_{x \in X}M_x / F \vDash \varphi(a_1 / {\equiv}_F, \dots, a_n / {\equiv}_F) \, \, \text{  if and only if  }\, \, \llbracket \varphi(a_1, \dots, a_n) \rrbracket \in F.
\]
Consequently, a positive sentence holds in a structure $M$  if and only if  it holds in some (equiv.\ every) prime power of $M$.
\end{Positive Los Theorem}

\begin{proof}
We recall that positive formulas are preserved by homomorphisms. Therefore, the assumption that $a_1, \dots, a_n \in S_F$ guarantees that $\llbracket \psi(a_1, \dots, a_n) \rrbracket$ is an upset of $\mathbb{X}$, for every positive formula $\psi$. We will use this fact without further notice.

We reason by induction on the construction of $\varphi$. In the base case, $\varphi$ is an atomic formula and the result holds by the definition of a prime product. The case where $\varphi = \psi_1 \land \psi_2$ follows from the inductive hypothesis and the fact that $F$ is a filter over $\mathbb{X}$. The case where $\varphi = \psi_1 \lor \psi_2$ follows from the inductive hypothesis and the fact that $F$ is prime. It only remains to consider the case where $\varphi = \exists w\? \psi(w, v_1, \dots, v_n)$. By the induction hypothesis we have
\[
\prod_{x \in X}M_x / F \vDash \varphi(a_1 / {\equiv}_F, \dots, a_n / {\equiv}_F) \, \, \text{  if and only if  }\, \, \text{there exists }b \in S_F \text{ s.t. }\llbracket \psi(b, a_1, \dots, a_n) \rrbracket \in F.
\]
Therefore, it only remains to prove that 
\[
\llbracket \exists w \?  \psi(w, a_1, \dots, a_n) \rrbracket \in F \, \, \text{  if and only if  }\, \, \text{there exists }b \in S_F \text{ s.t. }\llbracket \psi(b, a_1, \dots, a_n) \rrbracket \in F.
\]
For the sake of readability, we will write $E \coloneqq \llbracket \exists w \?  \psi(w, a_1, \dots, a_n) \rrbracket$. Since $E$ is an upset of $\mathbb{X}$ and $\llbracket \psi(b, a_1, \dots, a_n) \rrbracket \subseteq E$ for every $b \in S_F$, the implication from  right to left in the above display is straightforward. To prove the other implication, suppose that $E \in F$ and let $Y$ be the set of minimal elements of $E$. For every $y \in Y \subseteq E$ there exists $b(y) \in M_y$ such that
\begin{equation}\label{Eq : Positive Los 1}
M_y \vDash \psi(b(y), a_1(y), \dots, a_n(y)).
\end{equation}
As $\mathbb{X}$ is a wellfounded forest, for each $x \in E$ there exists a unique $y_x \in Y$ such that $y_x \leq x$. Therefore, we can define an element $b \in \prod_{x \in E}M_x$ as
\[
b(x) \coloneqq f_{y_x x}(b(y_x))
\]
for each $x \in E$. As $E \in F$, we have $b \in S_E \subseteq S_F$. Furthermore, from Condition (\ref{Eq : Positive Los 1}) and the fact that positive formulas are preserved by homomorphisms it follows
\[
M_x \vDash \psi(b(x), a_1(x), \dots, a_n(x)) \text{ for every }x \in E.
\]
Consequently, $E \subseteq \llbracket \psi(b, a_1, \dots, a_n) \rrbracket$. Since $E \in F$ and $\llbracket \psi(b, a_1, \dots, a_n) \rrbracket$ is an upset of $\mathbb{X}$, we conclude that $\llbracket \psi(b, a_1, \dots, a_n) \rrbracket \in F$ as desired. 
\end{proof}

\begin{Remark}
The proof of the case of the existential quantifier in the Positive \L os Theorem reveals why prime products have been defined for systems of structures indexed by wellfounded forests (as opposed to arbitrary posets).

We remark that the assumption that forests are wellfounded is necessary, as shown by the following example. For each integer $n$, let $M_n$ be the structure with universe $\mathbb{Z}$ whose language consists of a unary predicate $P$ interpreted as $\{ m \in \mathbb{Z} \mid m \leq n \}$. Let also $\{ M_n \mid n \in \mathbb{Z} \}$ be the ordered system indexed by $\mathbb{Z}$ whose homomorphisms the identity function $i \colon \mathbb{Z} \to \mathbb{Z}$. Furthermore, consider the prime filter $F \coloneqq \{ \mathbb{Z} \}$ over $\mathbb{Z}$. Lastly, let $\prod_{n \in \mathbb{Z}} M_n / F$ be the structure obtained from these ingredients using the instructions in the definition of a filter product. We have
\[
\llbracket \exists v \? P(v) \rrbracket = \mathbb{Z} \in F,
\]
while the interpretation of $P$ in $\prod_{n \in \mathbb{Z}} M_n / F$ is $\emptyset$ and, therefore, $\prod_{n \in \mathbb{Z}} M_n / F \nvDash \exists v \? P(v)$.
\qed
\end{Remark}

Since the notion of consequence is central to logic, the implications between positive formulas play also a fundamental role in positive model theory \cite{PoizatYeshkeyev}.

\begin{law}
 A formula $\varphi$ is said to be
\benroman
\item \emph{basic h-inductive} when there are two positive formulas $\psi_1$ and $\psi_2$ such that
    \[
   \varphi =  \forall v_1, \dots, v_n (\psi_1 \to \psi_2);
    \]
    \item \emph{h-inductive} when it is a conjunction of basic h-inductive formulas.
    \eroman
A set of h-inductive sentences will be called an \emph{h-inductive theory}.
\end{law}

As a consequence of the Positive \L os Theorem, we obtain the following:

\begin{Proposition}\label{Prop : prime products preserve h-inductive}
H-inductive sentences persist in prime products: if $\varphi$ is an h-inductive sentence and $\prod_{x \in X} M_x / F$ a prime product such that $M_x \vDash \varphi$ for every $x \in X$, then $\prod_{x \in X} M_x / F \vDash \varphi$.
\end{Proposition}

\noindent \textit{Proof.}
It suffices to prove the statement for the case where $\varphi$ is basic h-inductive. Then there are two positive formulas $\psi_1(v_1, \dots, v_n)$ and $\psi_2(v_1, \dots, v_n)$ such that $\varphi = \forall v_1, \dots, v_n\? (\psi_1 \to \psi_2)$. We will reason by contraposition. Suppose that $\prod_{x \in X} M_x / F \nvDash \varphi$. By the Positive \L os Theorem there are $a_1, \dots, a_n \in S_F$ such that
\[
\llbracket \psi_1(a_1, \dots, a_n) \rrbracket \in F \, \, \text{ and } \, \, \llbracket \psi_2(a_1, \dots, a_n) \rrbracket \notin F.
\]
Since $\llbracket \psi_2(a_1, \dots, a_n) \rrbracket$ is an upset of $\mathbb{X}$, this means that 
\[
\llbracket \psi_1(a_1, \dots, a_n) \rrbracket \nsubseteq \llbracket \psi_2(a_1, \dots, a_n) \rrbracket.
\]
Consequently, there exists $x \in V_{a_1} \cap \dots \cap V_{a_n}$ such that
\[
\pushQED{\qed} M_x \vDash \psi_1(a_1(x), \dots, a_n(x)) \land \lnot \psi_2(a_1(x), \dots, a_n(x)).\qedhere \popQED
\] 

\begin{law}
  A class of similar structures is said to be \emph{h-inductive} when it is closed under direct limits of chains of structures.
\end{law}

H-inductive theories and classes are related as follows (see, e.g., \cite[p.\ 108]{PoizatYeshkeyev}):

\begin{Theorem}\label{Thm : inductive classes are models of inductive theories}
  The class of models of an h-inductive theory is h-inductive.\
  Conversely, every elementary h-inductive class is axiomatized
  by an h-inductive theory.
\end{Theorem}

\noindent Notably, h-inductive elementaty classes can also be characterized in terms of prime products.

\begin{Corollary}\label{Cor : h-inductive classes via prime products}
A class of similar structures is elementary and h-inductive  if and only if  it is closed under isomorphisms, prime products, and ultraroots.
\end{Corollary}

\begin{proof}
We recall that a class of similar structures is elementary  if and only if  it is closed under isomorphisms, ultraproducts, and ultraroots \cite[Thm\ 2.13]{FrMoSc62}. Furthermore, such a class is h-inductive  if and only if  it is closed under direct limits of chains of structures. Since ultraproducts and  direct limits of chains of structures are special cases of prime products (see Examples \ref{Exa : reduced products} and \ref{Exa : limits of chains}), it follows that every class of similar closed under isomorphisms, prime products, and ultraroots is h-inductive and elementary. Conversely, every h-inductive elementary class is closed under isomorphisms and ultraroots because it is elementary and under prime products by Proposition \ref{Prop : prime products preserve h-inductive}.
\end{proof}

\section{Positive equivalence}

\begin{law}
The \emph{positive theory} of a structure $M$ is the set of positive sentences valid in $M$. Two structures are \emph{positively equivalent} when they have the same positive theory.
\end{law}

The Keisler Isomorphism Theorem states that, under the CGH, two structures are elementarily equivalent  if and only if  they have isomorphic ultrapowers \cite[Thm.\ 2.4]{Ke61}. As shown by Shelah, the result holds also without GCH \cite[p.\ 244]{She71}. The aim of this section is to establish the following:\footnote{It is an open problem whether one can dispense with GCH in the Positive Keisler Isomorphism Theorem too.}

\begin{Positive Keisler Isomorphism Theorem}\label{Positive Keisler Isomorphism Theorem}
Under \emph{GCH}, two structures are positively equivalent  if and only if  they have isomorphic prime powers of ultrapowers.
\end{Positive Keisler Isomorphism Theorem}

Before proving this result, we shall explain why it appears to be more complicated than the original Keisler Isomorphism Theorem. More precisely, the next examples show that
\benroman
\item A prime power of an ultrapower of a structure $M$ need not be isomorphic to any filter power of $M$;
\item Two positively equivalent structures need not have isomorphic prime powers.
\eroman

\begin{exa}\label{Exa : Q expended with constants}
Let $\mathbb{Q}$ be the poset of rational numbers with a constant for each element. Moreover, let $\mathbb{Q}_u$ be an ultrapower of $\mathbb{Q}$ containing an element $m$ such that $q < m$ for every $q \in \mathbb{Q}$. Then consider the endomorphism $f \colon \mathbb{Q}_u \to \mathbb{Q}_u$ defined by the rule
\[
f(p)=
\begin{cases}
m & \text{if $q < p$ for every $q \in \mathbb{Q}$;}\\
p & \text{otherwise.}
\end{cases}
\]
Lastly, let $\mathbb{Q}^\ast$ be the direct limit of the chain of structures
\[
\mathbb{Q}_u \xrightarrow{f} \mathbb{Q}_u \xrightarrow{f} \mathbb{Q}_u\xrightarrow{f} \cdots
\]

Recall from Example \ref{Exa : limits of chains} that direct limits of chains of structures are prime powers. Therefore, $\mathbb{Q}^\ast$ is a prime power of an ultrapower of $\mathbb{Q}$ by construction. We will prove that $\mathbb{Q}^\ast$ is not isomorphic to any filter power of $\mathbb{Q}$. To this end, observe that $\mathbb{Q}^\ast$ has a greatest element, namely, the image of $m$ in the direct limit. Therefore, it suffices to show that every nontrivial filter power of $\mathbb{Q}$ lacks a greatest element. Consider a nontrivial filter power  $\mathbb{Q}^X / F$ of $\mathbb{Q}$ and let $a \in S_F$. We need to find $b \in S_F$ such that $a / 	{\equiv}_{F} < b / 	{\equiv}_{F}$. Let $M$ be the set of minimal elements of $V_a$ and for each $x \in M$ let $q_x \in \mathbb{Q}$ be such that $a(x) < q_x$. Since $\mathbb{X}$ is a wellfounded forest, for each $y \in V_a$ there exists a unique $x_y \in M$ such that $x_y \leq y$. Because of this, the unique element $b \in \prod_{y \in V_a}M_y$ defined for every $y \in V_a$ as $b(y) \coloneqq f_{x_y y}(q_{x_y})$ belongs to $S_{V_a}$. Observe that the only endomorphism of $\mathbb{Q}$ is the identity function $i \colon \mathbb{Q} \to \mathbb{Q}$ because the language contains a constant for each element of $\mathbb{Q}$. Then for every $y \in V_a$,
\[
a(y) = f_{x_y y}(a(x_y)) = i(a(x_y)) = a(x_y) < q_{x_y} = i(q_{x_y}) = f_{x_y y}(q_{x_y}) = b(y).
\]
Therefore,
\[
\llbracket a = b \rrbracket = \emptyset \notin F \, \, \text{ and } \, \, \llbracket a \leq b \rrbracket = V_a \in F,
\]
where $\emptyset \notin F$ because $\mathbb{Q}^X / F$ is nontrivial. By the definition of a filter power we conclude that $a/ 	{\equiv}_{F} < b / 	{\equiv}_{F}$.
\qed
\end{exa}

The next example relies on the following observation (see, e.g., \cite[Thm.\ 3.1]{MorRafWan19PSC}):

\begin{Proposition}\label{Prop : positive theories and homomorphisms}
A structure $M$ satisfies the positive theory of a structure $N$  if and only if  there exist an ultrapower $M_u$ of $M$ and a homomorphism $f \colon N \to M_u$.
\end{Proposition}

\begin{exa}\label{Exa : positive equivalent without prime powers}
We will show that two positively equivalent structures need not have isomorphic prime powers. To this end, let $\mathbb{Q}$ be the structure defined in Example \ref{Exa : Q expended with constants} and $\mathbb{Q}^+$ the structure obtained by adding a greatest element $m$ to $\mathbb{Q}$. Clearly, $\mathbb{Q}$ is a substructure of $\mathbb{Q}^+$. Furthermore, $\mathbb{Q}^+$ is a substructure of the ultrapower $\mathbb{Q}_u$ of $\mathbb{Q}$ considered in Example \ref{Exa : Q expended with constants}. Therefore, $\mathbb{Q}$ and $\mathbb{Q}^+$ are positively equivalent by Proposition \ref{Prop : positive theories and homomorphisms}.

It only remains to prove that $\mathbb{Q}$ and $\mathbb{Q}^+$ do not have isomorphic prime powers. On the one hand, every filter power of $\mathbb{Q}$ induced by a proper filter (and, therefore, every prime power of $\mathbb{Q}$) lacks a greatest element, as shown in Example \ref{Exa : Q expended with constants}. On the other hand, every prime power $\mathbb{Q}^{+X} / F$ of $\mathbb{Q}^+$ has a greatest element, as we proceed to explain. Since the only endomorphism of $\mathbb{Q}^+$ is the identity function, the constant function $\hat{m} \colon X \to \mathbb{Q}^+$ with value $m$ is an element of $S_X \subseteq S_F$. Furthermore, for every $a \in S_F$ we have $\llbracket a \leq \hat{m} \rrbracket = V_a \in F$. By the definition of a filter product, we conclude that $a / {\equiv}_F \leq \hat{m}/ {\equiv}_F$. Hence, $\hat{m}/ {\equiv}_F$ is the greatest element of $\mathbb{Q}^{+X} / F$ as desired.
\qed
\end{exa}

The proof of the Positive Keisler Isomorphism Theorem relies on the next concept:

\begin{law}
A structure $M$ is said to be 
\benroman
\item \emph{positively} $\kappa$\emph{-saturated} for a cardinal $\kappa$ when for every $\vec{a} \in M^\lambda$ with $\lambda < \kappa$ and every set of positive formulas $p(x_1, \dots, x_n)$ with parameters in $\vec{a}$,
\[
\text{if }p \text{ is finitely satisfiable in }M \text{, then it is realized in }M;
\]
\item \emph{positively saturated} when it is positively $|M|$-saturated.
\eroman
\end{law}
\noindent While every saturated model is positively saturated, the converse need not hold in general (for instance, when viewed as a poset, the extended real number line is positively saturated, but not saturated). 

The proof of the next observation is a straightforward adaptation of the standard argument showing that saturated models are universal (see, e.g., \cite[Thm.\ 5.1.14]{ChKe90}).

\begin{Proposition}\label{Prop : saturation implies universality}
Let $M$ be a positively saturated structure. If $M$ satisfies the positive theory of a structure $N$ with $\vert N \vert \leq \vert M \vert$, there exists a homomorphism $f \colon N \to M$. 
\end{Proposition}

 We will also make use of the following result on classical saturation.

\begin{Theorem}\label{Thm : Keisler Kunen saturation}
The following hold for a structure $M$:
\benroman
\item\label{item : GCH saturation 1} For every cardinal $\kappa$ there exists a $\kappa$-saturated ultrapower of $M$;
\item\label{item : GCH saturation 2} Under GCH, if $M$ is infinite, it has arbitrarily large saturated ultrapowers.
\eroman
\end{Theorem}

\begin{proof}
Condition (\ref{item : GCH saturation 1}) follows from \cite[Thm.\ 2.1]{Keisler64IM} and \cite[Thm.\ 3.2]{Kunen72TAMS}. For Condition (\ref{item : GCH saturation 2}), see the proof of \cite[Cor.\ 2.3]{Keisler64IM}.
\end{proof}

\begin{Corollary}\label{Cor : saturated hom}
For every cardinal $\kappa$, if $M_u$ is an ultrapower of $M$ and $f \colon N \to M_u$ a homomorphism, there exists a $\kappa$-saturated ultrapower $M^\ast$ of $M$ with a homomorphism $g \colon N \to M^\ast$.
\end{Corollary}

\begin{proof}
By Theorem \ref{Thm : Keisler Kunen saturation}(\ref{item : GCH saturation 1}) there exists a $\kappa$-saturated ultrapower $M^\ast$ of $M_u$.\ As $M_u$ is an ultrapower of $M$ and ultrapowers of ultrapowers are still ultrapowers, we may assume that $M^\ast$ is an ultrapower of $M$. Furthermore, as $M_u$ embeds into $M^\ast$, we can view $f$ as a homomorphism from $N$ to $M^\ast$.
\end{proof}

The Positive Keisler Isomorphism Theorem is a consequence of the next observation:

\begin{Theorem} \label{thm:old}
Two structures $M_1$ and $M_2$ are positively equivalent  if and only if  there exists
\[
N \in \PPp\PPU(M_1) \cap \PPp\PPU(M_2).
\]
In addition, if each $M_i$ is positively saturated and either finite or of size $\geq \vert L \vert$, we can take
\[
N \in \PPp(M_1) \cap \PPp(M_2).
\]
\end{Theorem}

The next proof shows how to derive the Positive Keisler Isomorphism Theorem from the above result.

\begin{proof}
Consider two similar structures $M_1$ and $M_2$. If $M_1$ and $M_2$ have isomorphic prime powers of ultrapowers, then they are positively equivalent by the classical \L os Theorem and its positive version. Conversely, suppose that $M_1$ and $M_2$ are positively equivalent. By Theorem \ref{Thm : Keisler Kunen saturation}, under GCH each $M_i$ has a saturated ultrapower $M_i^\ast$. Furthermore, by the same theorem $M_i^\ast$ can be assumed to be either finite (if $M_i$ is finite) or of size $\geq \vert L \vert$ (if $M_i$ is infinite). Since $M_1^\ast$ and $M_2^\ast$ are also positively equivalent, we can apply Theorem \ref{thm:old} obtaining that there exists $N \in \PPp(M_1) \cap \PPp(M_2)$.
\end{proof}

The rest of the paper is devoted to proving Theorem \ref{thm:old}. To this end, we recall that a map $f \colon M \to N$  between two structures $M$ and $N$ is an \emph{immersion} when for every positive formula $\varphi(v_1, \dots, v_n)$ and $a_1, \dots, a_n \in M$,
\[
M \vDash \varphi(a_1, \dots, a_n) \, \, \text{  if and only if  }\, \, N \vDash \varphi(f(a_1), \dots, f(a_n)).
\]

\begin{law}
A model $M$ of a theory $T$ is said to be \emph{positively existentially closed} (\?\emph{pec}, for short) when every homomorphism from $M$ to a model of $T$ is an immersion.
\end{law}

We rely on the following description of pec models:

\begin{law}
 Let $\varphi(\vec{v}\?)$ be a positive formula and $T$ an h-inductive theory.
  The \emph{resultant} $\Res_T(\varphi)$ of $\varphi$ over $T$
  is the set of positive formulas $\psi(\vec{v}\?)$
  such that $T \vdash  \lnot \exists \vec{v} \?(\varphi \land \psi)$.
\end{law}

\begin{Proposition}[{\cite[Lem.~14]{BenYacPoiz07}}]\label{Prop : how to prove that smth is pec}
  A model $M$ of an h-inductive theory $T$ is pec  if and only if  for each positive formula $\varphi(\vec{v}\?)$ and $\vec{a} \in M$ such that $M \nvDash \varphi(\vec{a})$ there exists $\psi \in \Res_T(\varphi)$ such that $M \vDash \psi(\vec{a})$.
\end{Proposition}

The next two results are instrumental in constructing pec models.

\begin{Proposition}[{\cite[Thm.~1, Lem.\ 12]{BenYacPoiz07}}]\label{Prop : inductive + pec}
The following holds for an h-inductive theory $T$:
\benroman
\item\label{item:indpec:1} For every model $M$ of $T$ there exists a pec model $N$ of $T$ with a homomorphism $f \colon M \to N$;
\item\label{item:indpec:2} The class of pec models of $T$ is h-inductive.
\eroman
\end{Proposition}

\begin{Proposition}
  \label{prop:transfer}
  Let $M$ be a pec model of an h-inductive theory $T$ and $f \colon N \to M$ an immersion.  Then, $N$ is also a pec model of $T$.
\end{Proposition}
\begin{proof}
Since the theory $T$ is h-inductive, from the assumption that $f$ is an immersion and that $M \vDash T$ it follows that $N \vDash T$. To prove that $N$ is pec, we will use Proposition \ref{Prop : how to prove that smth is pec}. Consider a positive formula $\varphi(\vec{v}\?)$ and $\vec{a} \in N$ such that $N \nvDash \varphi(\vec{a})$. Since $f$ is an immersion and $\varphi$ positive, we obtain $M \nvDash \varphi(f(\vec{a}))$. As $M$ is a pec model of $T$, there exists $\psi(\vec{v}\?) \in \Res_T(\phi)$ such that $M \vDash \psi(f(\vec{a}))$. Since $f$ is an immersion and $\psi$ positive, this yields $N \vDash \psi(\vec{a})$ as desired.
\end{proof}

Lastly, we will rely on the following criterion for elementary equivalence.

\begin{Proposition}\label{Prop : elementary equivalence}
Two positively $\omega$-saturated pec models of an h-inductive theory are elementarily equivalent  if and only if  they have the same positive theory.
\end{Proposition}

\begin{proof}
See the paragraph immediately after the proof of \cite[Prop.\ 12]{PoizatYeshkeyev}.
\end{proof}

Positive equivalence is governed by the following concept.

\begin{law}
Given a structure $M$ with positive theory $T$, we let 
\[
\Th^+(M) \coloneqq T \cup  \{ \lnot \varphi \,(= \varphi \to \bot) \mid \varphi \text{ is a positive sentence and }M \nvDash \varphi \}.
\]
\end{law}

\noindent Notice that $\Th^+(M)$ is an h-inductive theory. Furthermore, two structures $M$ and $N$ are positively equivalent  if and only if  $\Th^+(M) = \Th^+(N)$.

\begin{Proposition}\label{Prop : main trick}
For every structure $M$ there exists a positively $\omega$-saturated pec model $M_\omega$ of $\Th^+(M)$ in $\Lim\PPU(M)$.\ In addition, if $M$ is positively saturated and either finite or such that $\vert L \vert \leq \vert M \vert$, we can take $M_\omega \in \Lim(M)$.
\end{Proposition}

\begin{proof}
We will define a chain of structures $\{ M_n \mid n \in \omega \}$ such that $M_n \vDash \Th^+(M)$ for each $n \in \omega$. First let $M_0 \coloneqq M$. Then suppose that the chain of structures $\{ M_m \mid m \leq n \}$ has already been defined. Since $M_n$ is a model of the h-inductive theory $\Th^+(M)$, by Proposition \ref{Prop : inductive + pec}(\ref{item:indpec:1}) there exists a pec model $M_n^\ast$ of $\Th^+(M)$ with a homomorphism $g_n \colon M_n \to M_n^\ast$. Furthermore, as $M_n^\ast$ is a model of $\Th^+(M)$, the structures $M_n^\ast$ and $M$ has the same positive theory.  In particular, $M$ satisfies the positive theory of $M_n^\ast$.  Therefore, we can apply Proposition \ref{Prop : positive theories and homomorphisms} obtaining an ultrapower $M_{n +1}$ of $M$ with a homomorphism $h_n \colon M_n^\ast \to M_{n +1}$. Since $M$ is a model of $\Th^+(M)$, so is the ultrapower $M_{n +1}$.\ In addition, the ultrapower $M_{n +1}$ can be assumed to be $\omega$-saturated in virtue of Corollary \ref{Cor : saturated hom}. Then we let $\{ M_m \mid m \leq n + 1 \}$ be the chain of structures obtained by extending $\{ M_m \mid m \leq n \}$ with $M_{n +1}$ and a homomorphism $f_{m \?\?\? n+1} \colon M_m \to M_{n+1}$ for each $m \leq n +1$ defined as follows:
\[
f_{m \?\?\? n+1} \coloneqq 
\begin{cases}
\text{the identity map }i \colon M_{n+1} \to M_{n+1} &\text{if }m = n +1;\\
h_{n}\circ g_{n}\circ f_{m \?\?\? n}& \text{otherwise}.\\
\end{cases}
\]
Lastly, let $M_\omega$ be the direct limit of the chain of structures $\{ M_n \mid n \in \omega \}$. Since $\Th^+(M)$ is an h-inductive theory and $M_n \vDash \Th^+(M)$ for each $n \in \omega$, we can apply Theorem \ref{Thm : inductive classes are models of inductive theories} obtaining that $M_\omega \vDash \Th^+(M)$.

We will prove that $M_\omega$ is a pec model of $\Th^+(M)$ that is  positively $\omega$-saturated and belongs to $\Lim\PPU(M)$. To this end, observe that $M_\omega$ is also the direct limit of
\benroman
\item\label{item : chain : pec} a chain of structures whose members are of the form $M_n^\ast$ for $n \in \omega$, and of
\item\label{item : chain : ultra} a chain of structures whose members are of the form $M_{n + 1}$ for $n \in \omega$.
\eroman

On the one hand, each $M_n^\ast$ is a pec model of $\Th^+(M)$ by construction. Therefore, from Condition (\ref{item : chain : pec}) it follows that $M_\omega$ is the direct limit of a chain of pec models of $\Th^+(M)$. Since the theory $\Th^+(M)$ is h-inductive, from Proposition \ref{Prop : inductive + pec}(\ref{item:indpec:2}) it follows that $M_\omega$ is a pec model of $\Th^+(M)$. On  the other hand, each $M_{n + 1}$ is an ultrapower of $M$. Therefore, from Condition (\ref{item : chain : ultra}) it follows $M_\omega \in \Lim\PPU(M)$.

It only remains to prove that $M_\omega$ is positively $\omega$-saturated. Recall that $M_\omega$ is the direct limit of the chain of structures $\{ M_n \mid n \in \omega \}$. For each $n \in \omega$ we will denote the canonical homomorphism from $M_n$ to $M_\omega$ associated with the direct limit $M_\omega$ by $f_n \colon M_n \to M_\omega$. Then consider $a_1, \dots, a_n \in M_\omega$ and let $p(v_1, \dots, v_n, a_1, \dots, a_n)$ be a set of positive formulas that is finitely satisfiable in $M_\omega$. Since $M_\omega$ is the direct limit of $\{ M_n \mid n \in \omega \}$, there exist $m \in \omega$ and $\hat{a}_1, \dots, \hat{a}_n \in M_{m}$ such that $f_m(\hat{a}_1) = a_1, \dots, f_m(\hat{a}_n) = a_n$. 

By the universal property of the direct limit we have
\[
f_m = f_{m+1} \circ f_{m \?\?\? m+1} = f_{m+1} \circ h_m \circ g_m.
\]
Thus,
\begin{align}\label{Eq : p after the homs}
    \begin{split}
p(v_1, \dots, v_n, a_1, \dots, a_n) &= p(v_1, \dots, v_n, f_m(\hat{a}_1), \dots, f_m(\hat{a}_n))\\
&=p(v_1, \dots, v_n, f_{m+1}(h_m(g_m(\hat{a}_1))), \dots, f_{m+1}(h_m(g_m(\hat{a}_n)))).
\end{split}
\end{align}

Since $M^\ast_m$ is a pec model of $\Th^+(M)$ and $M_{\omega}$ a model of $\Th^+(M)$, the homomorphism $f_{m+1} \circ h_m \colon M^\ast_m \to M_\omega$ is an immersion. Consequently, from the assumption that the set $p(v_1, \dots, v_n, a_1, \dots, a_n)$ is finitely satisfiable in $M_\omega$ and the above display it follows that $p(v_1, \dots, v_n, g_m(\hat{a}_1), \dots, g_m(\hat{a}_n))$ is finitely satisfiable in $M^\ast_m$. As positive formulas are preserved by homomorphisms, the set $p(v_1, \dots, v_n, h_{m}(g_m(\hat{a}_1)), \dots, h_{m}(g_m(\hat{a}_n)))$ is finitely satisfiable in $M_{m+1}$. Since $M_{m+1}$ is $\omega$-saturated, there exist $b_1, \dots, b_n \in M_{m+1}$ such that
\[
M_{m+1} \vDash p(b_1, \dots, b_n, h_{m}(g_m(\hat{a}_1)), \dots, h_{m}(g_m(\hat{a}_n))).
\]
Since positive formulas are preserved by homomorphisms, from Condition (\ref{Eq : p after the homs}) it follows that $M_\omega \vDash p(f_{m+1}(b_1), \dots, f_{m+1}(b_n), a_1, \dots, a_n)$. Hence, we conclude that $M_\omega$ is positively $\omega$-saturated.

To prove the second part of the statement, suppose that $M$ is positively saturated. If $M$ is finite, we have $M_n \cong M$ for each $n \in \omega$ because $M_n$ is an ultrapower of $M$. Therefore, the above construction yields $M_\omega \in \Lim(M)$ as desired. Then we consider the case where $M$ is infinite and $\vert L \vert \leq \vert M \vert$. Since $M$ is a model of $\Th^+(M)$, by Proposition \ref{Prop : inductive + pec}(\ref{item:indpec:1}) there exists a pec model $M^\ast$ of $\Th^+(M)$ with a homomorphism $g \colon M \to M^\ast$. As $M$ is infinite and such that $\vert L \vert \leq \vert M \vert$, by the downward L\"owenheim--Skolem Theorem there exists an elementary substructure $N$ of $M^\ast$ containing $g[M]$ such that $\vert N \vert \leq \vert M \vert$. Since $N$ is an elementary substructure of $M^\ast$ and $M^\ast$ is a pec model of the h-inductive theory $\Th^+(M)$, we can apply Proposition \ref{prop:transfer} obtaining that $N$ is also a pec model of $\Th^+(M)$. Therefore, we may assume without loss of generality that $M^\ast = N$ and, therefore, that $\vert M^\ast \vert \leq \vert M \vert$. As $M^\ast$ is a model of $\Th^+(M)$, we know that $M$ and $M^\ast$ have the same positive theory. Together with $\vert M^\ast \vert \leq \vert M \vert$ and the assumption that $M$ is positively saturated, this implies that there exists a homomorphism $h \colon M^\ast \to M$ by Proposition \ref{Prop : saturation implies universality}. Then we consider the endomorphism $f \coloneqq h \circ g$ of $M$. Then let $M_\omega$ be the direct limit of the chain of strucutres
\[
M \xrightarrow{f} M \xrightarrow{f} M \xrightarrow{f} \cdots
\]
The argument detailed above shows that $M_\omega$ is an $\omega$-saturated pec model of $\Th^+(M)$. Furthermore, $M_\omega \in \Lim(M)$ by construction.
\end{proof}

We are now ready to prove Theorem \ref{thm:old}.

\begin{proof}
From the classical \L os Theorem and its positive version it follows that if there exists $N \in \PPp\PPU(M_1) \cap \PPp\PPU(M_2)$, then $M_1$ and $M_2$ are positively equivalent. Conversely, suppose that $M_1$ and $M_2$ are positively equivalent. Then let
\[
T \coloneqq \Th^+(M_1) = \Th^+(M_2).
\]
By Proposition \ref{Prop : main trick} there are
\[
M_1^\ast \in \Lim\PPU(M_1)\, \, \text{ and } \, \, M_2^\ast \in \Lim\PPU(M_2)
\]
positively $\omega$-saturated pec models of $T$. Furthermore, $M^\ast_1$ and $M^\ast_2$ have the same positive theory by the classical \L os Theorem and its positive version. Consequently, $M_1^\ast$ and $M_2^\ast$ are elementarily equivalent by Proposition \ref{Prop : elementary equivalence}. In view of the Keisler-Shelah Isomorphism Theorem, there exists $N \in \PPU(M_1^\ast) \cap \PPU(M_2^\ast)$. Together with the above display, this yields
\[
N \in \PPU\Lim\PPU(M_1) \cap \PPU\Lim\PPU(M_2).
\]
By Proposition \ref{Prop : ultraproducts of chains} this simplifies to $N \in \PPp\PPU(M_1) \cap \PPp\PPU(M_2)$ as desired.

To prove the second part of the statement, suppose that $M_1$ and $M_2$ are positively saturated (in addition to positively equivalent) and that each $M_i$ is either finite or of size $\geq \vert L \vert$. By Proposition \ref{Prop : main trick} we can take 
\[
M_1^\ast \in \Lim(M_1)\, \, \text{ and } \, \, M_2^\ast \in \Lim(M_2)
\]
and repeat the argument above obtaining that $N \in \PPp(M_1) \cap \PPp(M_2)$.
\end{proof}

\begin{Remark}
In practice, the assumption of GCH in the Positive Keisler Isomorphism Theorem can sometimes be dispensed with. For instance, this is the case for each pair $M_1$ and $M_2$ of positively equivalent structures for which the h-inductive theory
\[
T \coloneqq\Th^+(M_1) = \Th^+(M_2)
\]
is \emph{bounded}, i.e., the size of each pec model of $T$ is $\leq \kappa$ for some cardinal $\kappa$ \cite{PoizatYeshkeyev}. To prove this, observe that by Theorem \ref{Thm : Keisler Kunen saturation}(\ref{item : GCH saturation 1}) each $M_i$ has a $\kappa$-saturated ultrapower $M_i^\ast$. Moreover, by Proposition \ref{Prop : inductive + pec}(\ref{item:indpec:1}) there exists a pec model $N_i$ of $T$ with a homomorphism $g_i \colon M_i^\ast \to N_i$. The assumption that $T$ is bounded guarantees that $\vert N_i \vert \leq \kappa$. Therefore, we can apply Proposition \ref{Prop : saturation implies universality} obtaining a homomorphism $h_i \colon N_i \to M_i^\ast$. Then consider the endomorphism $f_i \coloneqq h_i \circ g_i$ of $M_i^\ast$. Observe that the direct limit $M_i^+$ of the chain of structures
\[
M_i^\ast \xrightarrow{f} M_i^\ast \xrightarrow{f} M_i^\ast \xrightarrow{f} \cdots
\]
is a positively $\omega$-saturated pec model of $T$. From Proposition \ref{Prop : elementary equivalence} it follows that $M_1^+$ and $M_2^+$ are elementarily equivalent. Therefore, they have isomorphic ultrapowers by the Keisler-Shelah Isomorphism Thereom. Consequently, there exists $N \in \PPU\Lim(M_1^\ast) \cap \PPU\Lim(M_2^\ast)$. By Proposition \ref{Prop : ultraproducts of chains} this simplifies to $N \in \PPp(M_1^\ast) \cap \PPp(M_2^\ast)$. Hence, we conclude that $M_1$ and $M_2$ have isomorphic prime powers of ultrapowers. 
\qed
\end{Remark}

\begin{exa}[\textsf{Passive structural completeness}]
We close this paper with an application to algebraic logic.\ A quasivariety is said to be \emph{passively structurally complete} when all its nontrivial members have the same positive theory \cite{MorRafWan19PSC}. From a logical standpoint, the interest of this notion is justified as follows: when a propositional logic $\vdash$ is algebraized by a quasivariety $\mathsf{K}$ in the sense of \cite{BP89}, then all the vacuously admissible rules of $\vdash$ are derivable in $\vdash$  if and only if  $\mathsf{K}$ is passively structurally complete \cite[Fact 2, p.\ 68]{Wro09}. Both the Positive Keisler Isomorphism Theorem and Theorem \ref{thm:old} yield immediate descriptions of passive structurally complete quasivarieties in terms of prime powers.
\qed
\end{exa}

\paragraph{\bfseries Acknowledgements.}
We thank Tom\'{a}\v{s} Jakl and Guillermo Badia for helpful conversations and pointers. The first author was supported by the \textit{Beatriz Galindo} grant BEAGAL\-$18$/$00040$ funded by the Ministry of Science and Innovation of Spain. The second author's work was carried out within the project \emph{Supporting the
internationalization of the Institute of Computer Science of the Czech
Academy of Sciences} (number CZ.02.2.69/0.0/0.0/18\_053/0017594), funded by
the Operational Programme Research, Development and Education of the
Ministry of Education, Youth and Sports of the Czech Republic. The
project is co-funded by the EU.

\section*{Appendix}

\begin{proof}[Proof of Proposition \ref{Prop : ultraproducts of chains}]
We detail the proof of the inclusion $\PPU\Lim(\mathsf{K}) \subseteq \PPp(\mathsf{K})$, since the proof of the other inclusion regarding reduced and filter products is analogous (in fact, simpler). 

Let $\kappa$ be a cardinal and for each $\alpha < \kappa$ let $M_\alpha$ be the direct limit of a chain of structures $\{ M_x \mid x \in X_\alpha \}$ in $\mathsf{K}$ indexed by a well ordered poset $\mathbb{X}_\alpha$. Moreover, let $U$ be a ultrafilter over $\kappa$. We need to prove that
\[
\prod_{\alpha < \kappa} M_\alpha / U \in \PPp(\mathsf{K}).
\]

Without loss of generality, we may assume that the members of $\{\mathbb{X}_\alpha\}_{\alpha < \kappa}$
  are pairwise disjoint.
  Then let $\mathbb{X}$ be the poset obtained as the disjoint union of $\{\mathbb{X}_\alpha \}_{\alpha < \kappa}$ and define
\[
     F \coloneqq \{ V \in \mathsf{Up}(\mathbb{X}) \mid \{ \alpha < \kappa \mid  X_\alpha  \cap V \ne \emptyset \} \in U \}.
\]
\begin{Claim}  
The set $F$ is a prime filter over $\mathbb{X}$.
\end{Claim}
 
\begin{proof}[Proof of the Claim]
  Clearly, $F$ is an upset of $\mathsf{Up}(\mathbb{X})$.  Moreover, $F$ is nonempty as it contains $X$. To prove that $F$ is closed under binary intersections, consider $V, W \in F$. Then
\begin{equation}\label{Eq: F closed under meets 1}
\{ \alpha < \kappa \mid  X_\alpha  \cap V  \ne \emptyset \} \cap \{ \alpha < \kappa \mid  X_\alpha  \cap W  \ne \emptyset \} \in U.
\end{equation}
We will show that 
\begin{equation}\label{Eq: F closed under meets 2}
\{ \alpha < \kappa \mid  X_\alpha  \cap V  \ne \emptyset \} \cap \{ \alpha < \kappa \mid  X_\alpha  \cap W  \ne \emptyset \} \subseteq \{ \alpha < \kappa \mid  X_\alpha  \cap V \cap W \ne \emptyset \}.
\end{equation}
To this end, consider $\alpha < \kappa$ such that $X_\alpha  \cap V  \ne \emptyset$ and $X_\alpha  \cap W  \ne \emptyset$. Then there are $v \in X_\alpha \cap V$ and $w \in X_\alpha \cap W$. Since $\mathbb{X}_\alpha$ is linearly ordered, we may assume that $v \leq w$. Since $V \cap X_\alpha$ is an upset of $\mathbb{X}_\alpha$ (because $V$ is an upset of $\mathbb{X}$), this implies $w \in X_\alpha \cap V$. Thus, $w \in X_\alpha \cap V \cap W$ and, therefore, $X_\alpha \cap V \cap W \ne \emptyset$. From Conditions (\ref{Eq: F closed under meets 1}) and (\ref{Eq: F closed under meets 2}) it follows that 
\[
\{ \alpha < \kappa \mid  X_\alpha  \cap V \cap W \ne \emptyset \} \in U.
\]
Thus, $V \cap W \in F$ as desired. We conclude that $F$ is a filter over $\mathbb{X}$.

The definition of $F$ guarantees that $\emptyset \notin F$. Therefore, $F$ is proper. To prove that it is prime, consider $V, W \in \mathsf{Up}(\mathbb{X})$ such that $V \cup W \in F$. Then
\[
\{ \alpha < \kappa \mid X_\alpha \cap V \ne \emptyset \} \cup \{ \alpha < \kappa \mid X_\alpha \cap W \ne \emptyset \} = \{ \alpha < \kappa \mid X_\alpha \cap (V \cup W) \ne \emptyset \} \in U.
\]
Since $U$ is an ultrafilter over $\kappa$, it is also a prime filter over $\kappa$ ordered under the identity relation (Remark \ref{Rem: filters on X}). Therefore, from the above display it follows that
\[
\text{either }\{ \alpha < \kappa \mid X_\alpha \cap V \ne \emptyset \} \in U\text{ or } \{ \alpha < \kappa \mid X_\alpha \cap W \ne \emptyset \} \in U.
\]
This, in turn, implies that either $V$ or $W$ belongs to $F$.
\end{proof} 

  Now, recall that $\mathbb{X}$ is the disjoint union of the well ordered posets $\mathbb{X}_\alpha$. Therefore, the union $\{ M_x \mid x \in X \}$ of the ordered systems $\{ M_x \mid x \in X_\alpha \}$ is a well-defined ordered system indexed by a wellfounded forest. Furthermore, $F$ is a prime filter over $\mathbb{X}$ by the Claim. Consequently, we can form the associated prime product $\prod_{x \in X}M_x / F$. 
  
In order to conclude the proof, it suffices to prove that
\[
\prod_{\alpha < \kappa} M_\alpha / U \cong \prod_{x \in X} M_x / F.
\]  
To this end, observe that for every $\alpha < \kappa$ and $a \in M_\alpha$ there exist $z_{a} \in X_\alpha$ and $m_a \in M_{z_{a}}$ such that $f_{z_a}(m_a) = a$, where $f_{z_a} \colon M_{z_a} \to M_\alpha$ is the canonical homomorphism associated with the direct limit $M_\alpha$. For every each $a \in \prod_{\alpha < \kappa} M_\alpha$, let
\[
Y_{a} \coloneqq \{ x \in X \mid z_{a(\alpha)} \leq x \text{ for some }\alpha < \kappa \}. 
\]
Notice that each $x \in Y_{a}$ there exists exactly one $\alpha < \kappa$ such that $z_{a(\alpha)} \leq x$ because $\mathbb{X}$ is the disjoint union of the various $\mathbb{X}_\alpha$ and these are linearly ordered. We will denote this $\alpha$ by $\beta_{a x}$. Bearing this in mind, let $g(a)$ be the only element of $\prod_{x \in Y_{a}} M_x$ defined for every $x \in Y_{a}$ as
\[
g(a)(x) \coloneqq f_{z_{a(\beta_{ax})}x}(m_{a(\beta_{ax})}).
\]

\begin{Claim}\label{Claim : g(a) in S}
For every $a \in \prod_{\alpha < \kappa} M_\alpha$ we have $g(a) \in S_F$ and $V_{g(a)} = Y_a$.
\end{Claim}

\noindent \textit{Proof of the Claim.}
It suffices to show that $Y_a \in F$ and that for every $x, y \in Y_a$,
\[
x \leq y \text{ implies }f_{xy}(g(a)(x)) = g(a)(y).
\]
By definition $Y_a$ is an upset of $\mathbb{X}$ which, moreover, is nondisjoint with every $X_\alpha$ because $z_{a(\alpha)} \in X_\alpha \cap Y_a$. Together with the definition of $F$, this yields $Y_a \in F$. Then consider $x, y \in Y_a$ such that $x \leq y$. From $z_{a(\beta_{ax})} \leq  x \leq y$ and the fact that $\beta_{ay}$ is the unique $\alpha < \kappa$ such that $z_{a(\alpha)} \leq y$ it follows $\beta_{ax} = \beta_{ay}$. Therefore, by the definition of $g(a)$ we obtain
\[
g(a)(x) = f_{z_{a(\beta_{ax})} x}(m_{a(\beta_{ax})}) \, \, \text{ and } \, \, g(a)(y) = f_{z_{a(\beta_{ax})} y}(m_{a(\beta_{ax})}).
\]
Furthermore, from $z_{a(\beta_{ax})} \leq  x \leq y$ it follows $f_{z_{a(\beta_{ax})} y} = f_{xy} \circ f_{z_{a(\beta_{ax})} x}$. Together with the above display, this yields
\[
\pushQED{\qed} f_{xy}(g(a)(x)) = f_{xy}(f_{z_{a(\beta_{ax})} x}(m_{a(\beta_{ax})})) = f_{z_{a(\beta_{ax})} y}(m_{a(\beta_{ax})}) = g(a)(y).\qedhere \popQED
\]

Then we turn to prove the following:

\begin{Claim}\label{Claim : embedding well defined} For every atomic formula $\varphi(v_1, \dots, v_n)$ and $a_1, \dots, a_n \in \prod_{\alpha < \kappa} M_\alpha$,
\begin{equation}
\prod_{\alpha < \kappa} M_\alpha / U \vDash \varphi(a_1/ \er{U}, \dots, a_n/ \er{U}) \, \, \text{  if and only if  }\, \, \prod_{x \in X} M_x / F \vDash \varphi(g(a_1)/ \er{F}, \dots, g(a_n)/ \er{F}).
\end{equation}
\end{Claim}

\begin{proof}[Proof of the Claim]
We begin by showing that
\begin{align*}
\,&\{ \alpha < \kappa \mid M_\alpha \vDash \varphi(a_1(\alpha), \dots, a_n(\alpha)) \} \\
=\,& \{ \alpha < \kappa \mid M_\alpha \vDash \varphi(f_{z_{a_1(\alpha)}}(m_{a_1(\alpha)}), \dots, f_{z_{a_n(\alpha)}}(m_{a_n(\alpha)})) \}\\
=\,& \{ \alpha < \kappa \mid \text{there exists }x \geq z_{a_1(\alpha)}, \dots, z_{a_n(\alpha)}\text{ s.t. }M_x \vDash  \varphi(f_{z_{a_1(\alpha)}x}(m_{a_1(\alpha)}), \dots, f_{z_{a_n(\alpha)}x}(m_{a_n(\alpha)})) \}\\
=\,& \{ \alpha < \kappa \mid \text{there exists }x \geq z_{a_1(\alpha)}, \dots, z_{a_n(\alpha)}\text{ s.t. }M_x \vDash  \varphi(g(a_1)(x), \dots, g(a_n)(x)) \}\\
=\,& \{ \alpha < \kappa \mid \text{there exists }x \in X_\alpha \cap V_{g(a_1)} \cap \dots \cap V_{g(a_n)}\text{ s.t. }M_x \vDash  \varphi(g(a_1)(x), \dots, g(a_n)(x)) \}\\
=\,& \{ \alpha < \kappa \mid X_\alpha \cap \llbracket \varphi(g(a_1), \dots, g(a_n))\rrbracket \ne \emptyset \}.
\end{align*}
The first of the equalities above holds by the definition of $z_{a_i(\alpha)}$ and $m_{a_i(\alpha)}$, the second because $M_\alpha$ is the direct limit of the chain of structures $\{ M_x \mid x \in X_\alpha \}$, the third by the definition of $g(a_i)$, and the fifth by the definition of $\llbracket \varphi(g(a_1), \dots, g(a_n))\rrbracket$. To prove the fourth, it suffices to show that
\[
X_\alpha \cap V_{g(a_1)} \cap \dots \cap V_{g(a_n)} = \{ x \in X \mid x \geq z_{a_1(\alpha)}, \dots, z_{a_n(\alpha)}\}.
\]
The above equality, in turn, holds because $V_{g(a_i)} = Y_{a_i}$  by Claim \ref{Claim : g(a) in S} and $\mathbb{X}$ is the disjoint union of the various $\mathbb{X}_\beta$.

Observe that $\llbracket \varphi(g(a_1), \dots, g(a_n))\rrbracket$ is an upset of $\mathbb{X}$ (because positive formulas are preserved by homomorphisms). Therefore, from the above series of equalities and the definition of $F$ it follows that
\[
\{ \alpha < \kappa \mid M_\alpha \vDash \varphi(a_1(\alpha), \dots, a_n(\alpha)) \} \in U \, \, \text{  if and only if  }\, \,\llbracket \varphi(g(a_1), \dots, g(a_n))\rrbracket \in F.
\]
By the classical \L os Theorem and its positive version this yields the desired result.
\end{proof}

In view of Claim \ref{Claim : embedding well defined}, the map
\[
\hat{g} \colon \prod_{\alpha < \kappa} M_\alpha / U \to \prod_{x \in X} M_x / F
\]
defined by the rule $\hat{g}(a / \er{U}) \coloneqq g(a) / \er{F}$ is a well-defined embedding. Therefore, to prove that $\hat{g}$ is an isomorphism, it only remains to show that it is surjective. 

To this end, consider $a \in S_F$. For each $\alpha < \kappa$ and $x \in X_\alpha \cap V_a$ let $f_x \colon M_x \to M_\alpha$ be the canonical homomorphism associated with the direct limit $M_\alpha$. For each $\alpha < \kappa$ such that $X_\alpha \cap V_a \ne \emptyset$ there exists some $y_\alpha \in X_\alpha \cap V_a$ such that
\begin{equation}\label{Eq : sujectivity finally}
z_{f_{y_{\alpha}}(a(y_\alpha))} \leq y_a \, \, \text{ and } \, \, a(y_\alpha) = f_{z_{f_{y_{\alpha}}(a(y_\alpha))} \?\? y_\alpha}(m_{f_{y_{\alpha}}(a(y_\alpha))}).
\end{equation}
This is a consequence of the definition of the direct limit $M_\alpha$ and of the assumption that $X_\alpha \cap V_a$ is an upset of the linearly ordered poset $\mathbb{X}_\alpha$.

We define an element $b \in \prod_{\alpha < \kappa} M_\alpha$ as follows: for each $\alpha < \kappa$,
\[
b(\alpha) \coloneqq 
\begin{cases}
f_{y_{\alpha}}(a(y_\alpha)) &\text{if }X_\alpha \cap V_a \ne \emptyset;\\
\text{an arbitrary element of }M_\alpha &\text{otherwise}.\\
\end{cases}
\]
Our aim is to prove that $\hat{g}(b / {\equiv}_U) = a / {\equiv}_F$, i.e., $\llbracket g(b) = a \rrbracket \in F$. Since $g(b) \in S_F$, we know that $\llbracket g(b) = a \rrbracket$ is an upset of $\mathbb{X}$. Therefore, by the definition of $F$ it suffices to show that
\[
\{ \alpha < \kappa \mid X_\alpha \cap \llbracket g(b) = a \rrbracket \ne \emptyset \} \in U.
\]
As $a \in S_F$, we know that $V_a \in F$ and, therefore, that $\{ \alpha < \kappa \mid X_\alpha \cap V_a \ne \emptyset \} \in U$. Therefore, our task reduces to that of proving the inclusion
\[
\{ \alpha < \kappa \mid X_\alpha \cap V_a \ne \emptyset \} \subseteq \{ \alpha < \kappa \mid X_\alpha \cap \llbracket g(b) = a \rrbracket \ne \emptyset \}.
\]
Accordingly, let $\alpha < \kappa$ be such that $X_\alpha \cap V_a \ne \emptyset$. Then $y_\alpha \in X_a \cap V_a$. Furthermore, $b(\alpha) = f_{y_{\alpha}}(a(y_\alpha))$ by the definition of $b$. Therefore, from Condition (\ref{Eq : sujectivity finally}) it follows $z_{b(\alpha)} \leq y_\alpha$. By the definition of $Y_{b}$ this amounts to $y_\alpha \in Y_b$. Since $Y_b = V_{g(b)}$ by Claim \ref{Claim : g(a) in S}, we conclude that $y_\alpha \in V_{g(b)}$. Thus, $y_\alpha \in V_{g(b)} \cap V_a$. In addition, from the definition of $g(b)$ and Condition (\ref{Eq : sujectivity finally}) it follows
\[
g(b)(y_\alpha) = f_{z_{f_{y_{\alpha}}(a(y_\alpha))} \?\? y_\alpha}(m_{f_{y_{\alpha}}(a(y_\alpha))}) = a(y_\alpha).
\]
Therefore, $y_\alpha \in \llbracket g(b) = a \rrbracket$. Hence, we conclude that $X_\alpha \cap \llbracket g(b) = a \rrbracket \ne \emptyset$ as desired.
\end{proof}

\bibliographystyle{plain}

\enlargethispage{2em}

\end{document}